\documentclass[a4paper,11pt]{article}
\usepackage[utf8]{inputenc}

\usepackage{amsthm,amsmath,amsfonts,amssymb}
\usepackage[english]{babel}
\usepackage{graphicx}
\usepackage{verbatim,xcolor}

\newcommand{\ka}{\mathrm{Ka}}

\DeclareMathOperator{\dist}{dist}

\def\vec0{\mbox{\boldmath $0$}}

\theoremstyle{plain}   

\newtheorem{theorem}{Theorem}[section]
\newtheorem{proposition}[theorem]{Proposition}
\newtheorem{corollary}[theorem]{Corollary}
\newtheorem{lemma}[theorem]{Lemma}
\newtheorem{definition}[theorem]{Definition}

\newtheorem{question}{Question}

\newtheorem{problem}[theorem]{Problem}

\title{On mixed radial Moore graphs of diameter $3$}
\author{J.M. Ceresuela$^{a,b}$, N. L\'opez$^a$, D. Chemisana$^b$. 
\\ \\
{\small $^a$Dep. de Matem\`atica, Universitat de Lleida, Lleida (Spain)} \\
{\small $^b$Applied Physics Section of the Environm. Sci. Dept. Universitat de Lleida, Lleida (Spain)} \\
{\small {\tt \{jesusmiguel.ceresuela,nacho.lopez,daniel.chemisana\}@udl.cat}}
}
\date{}

\begin{document}

\maketitle

\begin{abstract}
Radial Moore graphs and digraphs are extremal graphs related to the Moore ones where the distance-preserving spanning tree is preserved for some vertices. This leads to classify them according to their proximity to being a Moore graph or digraph. In this paper we deal with mixed radial Moore graphs, where the mixed setting allows edges and arcs as different elements. An exhaustive computer search shows the top ranked graphs for an specific set of parameters. Moreover, we study the problem of their existence by providing two infinite families for different values of the degrees and diameter $3$. One of these families turns out to be optimal.   
\end{abstract}


\section{Introduction}


Given the values of the maximum out-degree $z$ and the diameter $k$, there is a natural upper bound $n_{z,k}$ for the largest order of a digraph with these two parameters,
\[
n(z,k)=1+z+\cdots + z^k,
\]
referred to as the {\em Moore bound\/} for digraphs. Digraphs attaining such a bound are called {\em Moore digraphs\/}. In particular, all vertices of a Moore digraph have the same degree ($d$) and the same eccentricity ($k$). It is well known that Moore digraphs do only exist in the trivial cases, $d=1$ or $k=1$, which correspond to the directed cycle of order $k+1$ and the complete digraph of order $d+1$, respectively (see \cite{plesnik74,BT80}). This has led to the study of digraphs `close' to the Moore ones. One way to do it is by allowing the existence of vertices with eccentricity just one more than the value they should have. In this context, regular digraphs of degree $z$, radius $k$, diameter at most $k+1$ and order equal to $n(z,k)$ are known as {\em radial Moore digraphs\/}. These extremal digraphs were first studied by Knor \cite{485578} and they exist for any value of $z$ and $k$ (see also Gimbert and L\'opez \cite{GL2008}). \\

Lately, the definition of radial Moore digraph was extended to undireted graphs. In this undirected case, given the values of the maximum degree $r$ and the diameter $k$ of a graph, there Moore bound is given by $N(r,k)$ where
\begin{equation}\label{eq:moore_bound}
N(r,k)=1+r+r(r-1)+\cdots+r(r-1)^{k-1},
\end{equation}
Graphs attaining such a bound are referred to as {\em Moore graphs\/}. From its definition, all vertices of a Moore graph have the same degree ($d$) and the same eccentricity ($k$). So, the definition of {\em radial Moore graphs\/} as regular graphs of degree $r$, radius $k$, diameter $k+1$ and order equal to $N(r,k)$ is just an undirected version of the same definition for digraphs. \\

Curiously enough, the undirected version of this problem has proved to be much more difficult and it is not obvious whether or not there exists a radial Moore graph for all possible values of degree and diameter. First Capdevila et al. \cite{CCEGL2010} proved the existence of radial Moore graphs for radius $2$ and any degree. They also considered some ranking measures of how well a radial Moore graph approximates a Moore graph. Later on, Exoo et al. \cite{EXOO20121507} gave a construction of radial Moore graphs of diameter $3$ for all degrees $d \geq 22$ and some other sporadic values of $k$ and $d$. In 2015, Gomez and Miller \cite{GOMEZ201515} presented a construction technique that produces radial Moore graphs for every value of diameter and degree large enough (depending on the diameter). Nevertheless, there exist infinitely many values of $r$ and $k$ for which the problem of existence of radial Moore graphs remains open.

\subsubsection*{Terminology and notation}

A {\em mixed} (or {\em partially directed\/}) graph $G$ with vertex set $V$ may contain a set $E$ of (undirected) {\em edges} as well a set $A$ of directed edges (also known as {\em arcs}). From this point of view, a {\em graph} [resp. {\em directed graph} or {\em digraph}] has all its edges undirected [resp. directed]. The set of vertices which
are adjacent from [to] a given vertex $v$ is denoted by $\Gamma^{+}(v)$ [$\Gamma^{-}(v)$]. The {\em undirected degree} of a vertex $v$, denoted by $d(v)$ is the number of edges incident to $v$. The {\em out-degree} [resp. {\em in-degree}] of vertex $v$, denoted by $d^+(v)$ [resp. $d^-(v)$], is the number of arcs emanating from [resp. to] $v$.  If $d^+(v)=d^-(v)=z$ and $d(v)=r$, for all $v \in V$, then $G$ is said to be {\em totally regular\/} of degrees $(r,z)$ (or simply {\em $(r,z)$-regular}).
A {\em walk\/} of length $\ell\geq 0$ from $u$ to $v$ is a sequence of $\ell+1$ vertices, $u_0u_1\dots u_{\ell-1}u_\ell$, such that $u=u_0$, $v=u_\ell$ and each pair $u_{i-1}u_i$, for $i=1,\ldots,\ell$, is either an edge or an arc of $G$. A {\em directed walk} is a walk containing only arcs. An {\em undirected walk} is a walk containing only edges. A walk whose vertices are all different is called a {\em path}.
The length of a shortest path from $u$ to $v$ is the {\it distance\/} from $u$ to $v$, and it is denoted by $\dist(u,v)$. Note that $\dist(u,v)$ may be different from $\dist(v,u)$, when shortest paths between $u$ and $v$ involve arcs. The sum of all distances from a vertex $v$, $s(v)=\sum_{u\in V} d(v,u)$, is referred to as the {\em status\/} of $v$ (see \cite{BuckHara}). We define the {\em status vector\/} of $G$, $\mathbf{s}(G)$, as the vector constituted by the status of all its vertices. Usually, when the vector is long enough, we denote it with a short description using superscripts, that is, $\mathbf{s}(G):s_1^{n_1},s_2^{n_2},\dots,s_k^{n_k}$, where $s_1>s_2> \dots >s_k$, and $n_i$ denotes the number of vertices having $s_i$ as its local status, for all $1 \leq i \leq k$. 
The {\em out-eccentricity\/} of a vertex $u$ is the maximum distance from $u$ to any vertex in $G$. A {\em central vertex} is a vertex having minimum out-eccentricity. The maximum distance between any pair of vertices is the {\it diameter} of $G$.

Mixed graphs can be seen as a generalization of both, undirected and directed graphs, see for instance the works by Nguyen and Miller \cite{nm08} and Buset et al. \cite{BUSET20162066}.

\subsubsection*{Organization of the paper}

This paper is organized as follows: In section \ref{sec:mrm} the reader will find a generalization of radial Moore graphs to the mixed case. A closeness measure is also presented there. We find all mixed radial Moore graphs for the case of $(r,z,k)=(2,1,2)$, and we present the closest ones to being a mixed Moore graph in section \ref{sec:ranking}. Next, in section \ref{sec:exist}, we study the problem of the existence of mixed radial Moore graphs of diameter $3$. To this end, we provide two infinite families $H_r$ (for $r\geq 1$ and $z=1$) and $G_z$ (for $r=1$ and $z \geq 1$). We use an algebraic method to design $H_r$. Besides, $G_z$ is constructed performing a convinient swap to a pair of arcs of the family of Kautz digraphs. Next we prove the optimality of this family. Some open problems and concluding remarks are presented in final section \ref{sec:remark}.


\section{Mixed radial Moore graphs}\label{sec:mrm}
\label{sec:measure}

The degree/diameter problem for mixed graphs asks for the largest possible number of vertices $n(r,z,k)$ in a mixed graph with maximum undirected degree $r$, maximum directed out-degree $z$, and diameter $k$. A natural upper bound for $n(r,z,k)$ is derived by counting the number of vertices at every distance from any given vertex $v$ in a mixed graph with given maximum undirected degree $r$, maximum directed out-degree $z$, and diameter $k$. This bound is known as the {\em Moore bound} for mixed graphs (see \cite{BUSET20162066} and \cite{DALFO20182872}):
\begin{equation}
\label{eq:moorebound4}
M(r,z,k)=A\frac{\lambda_1^{k+1}-1}{\lambda_1-1}+B\frac{\lambda_2^{k+1}-1}{\lambda_2-1},
\end{equation}
where
%
\begin{align}
\tau   &=(z+r)^2+2(z-r)+1, \label{v} \\
\lambda_1 &=\frac{1}{2}(z+r-1-\sqrt{\tau}),\qquad \lambda_2=\frac{1}{2}(z+r-1+\sqrt{\tau}),\label{u's} \\
A &=\frac{\sqrt{\tau}-(z+r+1)}{2\sqrt{v}},
\qquad
B =\displaystyle{\frac{\sqrt{\tau}+(z+r+1)}{2\sqrt{\tau}}}. \label{A's}
\end{align}

It is a matter of routine to check that $M(0,z,k)=n(z,k)$ and $M(r,0,k)=N(r,k)$. So the mixed Moore bound is a generalization of both undirected Moore bound and directed Moore bound, as expected. {\em Mixed Moore graphs} are those with order attaining \eqref{eq:moorebound4}, which means that between any pair of vertices there is a unique shortest path of length not greater than the diameter. Mixed Moore graphs must be totally regular of degree $(r,z)$ (see Bos\'{a}k~\cite{B79}) and the distance-preserving spanning tree, also known as Moore tree, hanging at any vertex $v$ does not depend on the chosen vertex. (see Fig. \ref{fig:arbolgen}).

\begin{figure}[ht]
        \centerline{\includegraphics*[width=\textwidth]{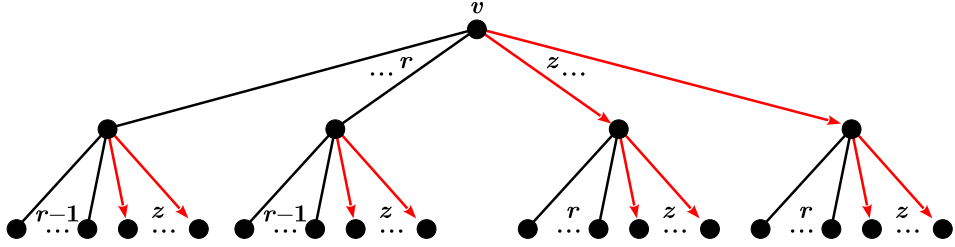}}
        \caption{Moore tree of a radial Moore graph of diameter $2$.}
        \label{fig:arbolgen}
\end{figure}
Next, we define a mixed radial Moore graph as a natural generalization of both radial Moore graph and radial Moore digraph.
\begin{definition}
A totally $(r,z)$-regular graph of radius $k$, diameter $k+1$ and order $M(r,z,k)$ is a $(r,z,k)$-mixed radial Moore graph.
\end{definition}
Notice that a $(0,z,k)$-mixed radial Moore graph is simply a radial Moore digraph and a $(r,0,k)$- mixed radial Moore graph is a radial Moore graph. The vertex set of any $(r,z,k)$-mixed radial Moore can be partitioned in two sets where one set contains all central vertices (those vertices with out-eccentricity $k$, that is, their corresponding distance preserving spannig tree is a Moore tree) and the other set has all non-central vertices (with out-eccentricity $k+1$). \\

Mixed radial Moore graphs can be seen as an approximation of mixed Moore graphs where the diameter constraint has been relaxed a bit. This relaxation may produce a significative number of graphs (as we will see later) that we want to rank in order to see which one is the 'closest' graph to being a Moore graph. The number of central vertices could be a first attempt to measure this closeness, but, as we will see later, this measure is not finer since there are many mixed radial Moore graphs with the same number of central vertices. \\ 

Two different ranking measures were introduced in \cite{CCEGL2010} for the undirected case. One ranking measure is related to the girth, since Moore (undirected) graphs attain the minimum order allowed by the girth. But this measure is no longer useful for the mixed case where short directed cycles appear in mixed Moore graphs. The other ranking measure is based on the status of a vertex in a  Moore graph. First we give a characterization of a mixed Moore graph in terms of the statuses of their vertices.

\begin{proposition}
Given three positive integers $r>1$, $z>1$ and $k>1$, let $G$ be a $(r,z,k)$-mixed radial Moore graph. Then, for every vertex $v$ of $G$ we have
\[
s(v) \geq A\left(\frac{k\lambda_1^{k+2}-(k+1)\lambda_1^{k+1}+\lambda_1}{(\lambda_1-1)^2}\right)+B\left(\frac{k\lambda_2^{k+2}-(k+1)\lambda_2^{k+1}+\lambda_2}{(\lambda_2-1)^2}\right)
\] 
where $\lambda_1,\lambda_2,A,B$ are defined in equations \eqref{u's} and \eqref{A's}. Moreover, this bound is attained for every vertex if and only if $G$ is a mixed Moore graph.   
\end{proposition}
\begin{proof}
The distance-preserving spanning tree of any central vertex $v$ in $G$ is completely determined (see Figure \ref{fig:arbolgen} for the case $k=2$). Let $N_i$ be the number of vertices at distance $i$ from $v$, with $N_i=R_i+Z_i$, where $R_i$ is the number of vertices that, in the corresponding tree rooted at $v$,  have an edge with their parents; and  $Z_i$ is the number of vertices that have an arc from their parents. Then,
\begin{equation}
\label{Ni}
N_i = R_i+Z_i = R_{i-1}((r-1)+z)+Z_{i-1}(r+z).
\end{equation}
Besides,
$Z_i=z(N_{i-1}-Z_{i-1})+zZ_{i-1}=zN_{i-1}$ and, hence,
\begin{align}
\label{recurNi}
N_i & =(r+z)N_{i-1}-R_{i-1}=(r+z)N_{i-1}-(N_{i-1}-Z_{i-1}) \nonumber\\
    & =(r+z-1)N_{i-1}+zN_{i-2},\qquad i=2,3,\ldots
\end{align}
with initial values $N_0=1$ and $N_1=r+z$. Hence, the status of vertex $v$ is given by
\[
s(v)=\sum_{i=0}^{k} iN_i
\]
Taking into account that $N_i$ is defined by the linear recurrence relation \eqref{recurNi}, we have that $N_i=A\lambda_1^i+B\lambda_2^i$. Hence,
\[
\sum_{i=0}^{k} iN_i=A\sum_{i=0}^{k} i\lambda_1^i+B\sum_{i=0}^{k} i\lambda_2^i
\]
Denote $S=\sum_{i=0}^{k} i\lambda_1^i$. Notice that 
\[
\lambda_1S-S=k\lambda_1^{k+1}-\sum_{i=1}^k \lambda_1^i =k\lambda_1^{k+1}-\frac{\lambda_1^{k+1}-1}{\lambda_1-1}+1
\]
As a consequence,
\[
S= \frac{k\lambda_1^{k+2}-(k+1)\lambda_1^{k+1}+\lambda_1}{(\lambda_1-1)^2} 
\]

The calculation for $\sum_{i=0}^{k} i\lambda_2^i$ is equivalent, obtaining,
\begin{equation}
s(v) = A\left(\frac{k\lambda_1^{k+2}-(k+1)\lambda_1^{k+1}+\lambda_1}{(\lambda_1-1)^2}\right)+B\left(\frac{k\lambda_2^{k+2}-(k+1)\lambda_2^{k+1}+\lambda_2}{(\lambda_2-1)^2}\right) 
\label{eq:statusMoore}
\end{equation}

A non-central vertex $u$ of $G$ has vertices at distance $k+1$, so $s(u) \geq s(v)$, completing the proof.
\end{proof}

From now on, let $\mathbf{s}_{r,z,k}$ be the vector of dimension $n=M(r,z,k)$ whose components are all equal to $s(v)$, where $s(v)$ denotes de status of any vertex in a mixed Moore graph (see Eq. \eqref{eq:statusMoore}). Notice that $\mathbf{s}_{r,z,k}$ represent the status vector of a mixed Moore graph. \\

Let us denote by ${\cal RM}(r,z,k)$ the set of all nonisomorphic $(r,z,k)$-mixed radial Moore graphs. Our purpose is to rank each mixed graph of this `population' in terms of how close is of being a mixed Moore graph, generalizing one ranking measure given in \cite{CCEGL2010} for the undirected case.

Let $G\in {\cal RM}(r,z,k)$. For every positive integer $p$ we define 
\[
N_p(G)=\|\mathbf{s}(G)-\mathbf{s}_{r,z,k}\|_p.
\]
In particular, $N_1(G)$ measures the difference between the total status of $G$ and the one corresponding to a mixed Moore graph. 

Given two mixed graphs $G_1,G_2\in {\cal RM}(r,z,k)$, we define $G_1$ and $G_2$ to be {\em status-equivalent\/}, $G_1 \sim G_2$, if they have the same status vector. In the quotient set of ${\cal RM}(r,z,k)$ by $\sim$, ${\cal RM}(r,z,k)/\sim$, we will say that $G_1$ is closer than $G_2$ to be a mixed Moore graph if there exists a positive integer $l$ such that
\[
N_p(G_1)=N_p(G_2),\ p=1,\dots,l-1\quad \mathrm{and}\quad  
N_{l}(G_1)<N_{l}(G_2),
\]
in which case it is denoted by $G_1 < G_2$. Notice that this relation induces a total order in ${\cal RM}(r,z,k)/\sim$, since
\[
\mathbf{s}(G_1)=\mathbf{s}(G_2) \iff N_p(G_1)=N_p(G_2), \  \mbox{for every}\ p=1,\dots,n.
\]

\section{Ranking of $(2,1,2)$-mixed radial Moore graphs.}\label{sec:ranking}

To find all mixed radial Moore graphs for $r=2$, $z=1$ and $k=2$, we begin by constructing the mixed Moore tree anchored at a distinguished vertex $0$. It has $11$ vertices. It remains to consider the arcs and the edges joining vertices at distance $2$ from $0$ (see figure \ref{fig:MT}). The directed part of the Moore tree can be completed by a bijection between the set of vertices with zero outdegree $\{4,5,6,7,8,9,10\}$ and the set of vertices with zero indegree $\{0,1,3,4,6,8,10\}$. Of course, there are some bijections that are not valid in order to complete the directed part (for instance, those containing fixed points or the assignation $5 \rightarrow 1$, among others). However, the number of bijections to be taken into account at this step is $< 7!$. Once a valid bijection $b$ is found, we proceed with the exhaustive search to complete the undirected part. To this end we consider the multiset of vertices $\{4,5,5,6,7,7,8,9,9,10\}$ with undirected degree less than $2$, taking into account that those vertices with zero undirected degree appears twice in the set. By considering all possible non ordered pairs of elements of this multiset we get the remainig edges of the mixed graph. Again some of these combinations will be not valid (for instance the ones containing the unordered pair $\{5,5\}$). A combination turns out to be fair depending also on the bijection $b$, since if some unordered pair matches with an element of  $b$ then such combination is dismissed. The total number of pairs of sets of directed arcs and undirected arcs to check is at most $7!\binom{10}{2,2,2,2,2}=\frac{10!7!}{2!^5}$, but due to the restrictions, only about $10^6$ are valid. Then, these sets are added to the Moore tree and the diameter of the resulting graph is computed. There are about $3\cdot 10^5$ mixed graphs with diameter $3$, but only 9486 are not isomorphic.

\begin{figure}[ht]
        \centerline{\includegraphics*[scale=0.5]{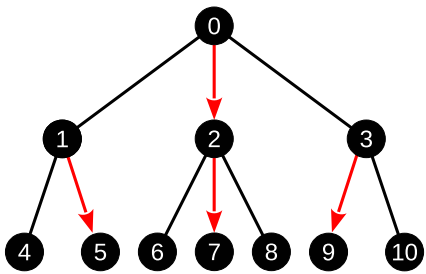}}
        \caption{Moore tree used as a basis to construct every $(2,1,2)$-mixed radial Moore graph.}
        \label{fig:MT}
\end{figure}

It turns out that ${\cal RM}(2,1,2)$ contains exactly $9486$ mixed graphs. According to their number of central vertices $c$, they are distributed as indicated in Table \ref{table:center_girth}. 

\begin{table}[htb]
\begin{center}
\begin{tabular}{|l|l|l|l|l|}\hline
$c=1$ & $c=2$ & $c=3$ & $c=4$ & Total \\ \hline
8529 (89.91\%) & 906 (9.55\%) & 13 (0.14\%) & 38 (0.40\%)&  9486 (100\%)
\\
\hline
\end{tabular}
\caption{Distribution of mixed radial Moore graphs according to their number of central vertices.} \label{table:center_girth}
\end{center}
\end{table}

Notice that the number of central vertices only takes four distinct values being $4$ their maximum.  It would be interesting to find an upper bound for the cardinality of the center of a mixed radial Moore graph as a fraction of its order. On the other hand, we remark that around $90\%$ of the mixed graphs have just one central vertex.

The first two mixed graphs $G_1$ and $G_2$ in the status ranking are shown in Figure \ref{fig:three_best}. Both graphs have the same status sequence: $\mathbf{s}(G_1): 17^4,18^6,19^1$, so $G_1 \sim G_2$. They attain the minimum value for the status norm $N_1(G_1)=8$. The undirected part of both graphs is the union of two cycles of lengths $5$ and $6$, but their directed part is different: $G_1$ has three directed cycles of lengths $4,4$ and $3$, meanwhile $G_2$ has two directed cycles of lengths $7$ and $4$. Curiously enough, $G_1$ has a non trivial automorphism (as it can be seen in figure \ref{fig:three_best}, where $G_1$ shows an axial symmetry) but $G_2$ has none. The spectrum of both $G_1$ and $G_2$ is $\{3^1,1^2,(-1)^2,\lambda_1^3,\lambda_2^3\}$, where $\lambda_1=\frac{-1+\sqrt{5}}{2}$ and $\lambda_2=\frac{-1-\sqrt{5}}{2}$. The set of eigenvalues $\{3,\lambda_1,\lambda_2\}$ is precisely the one that a mixed Moore graph should have in this case. $G_2$ is cospectral with $G_1$, and it can be obtained by applying  a recent method to obtain cospectral digraphs with  a locally line digraph. For more details, see Dalf\'o and Fiol \cite{DALFO201652}.

\begin{figure}[ht]
        \centerline{\includegraphics*[width=\textwidth]{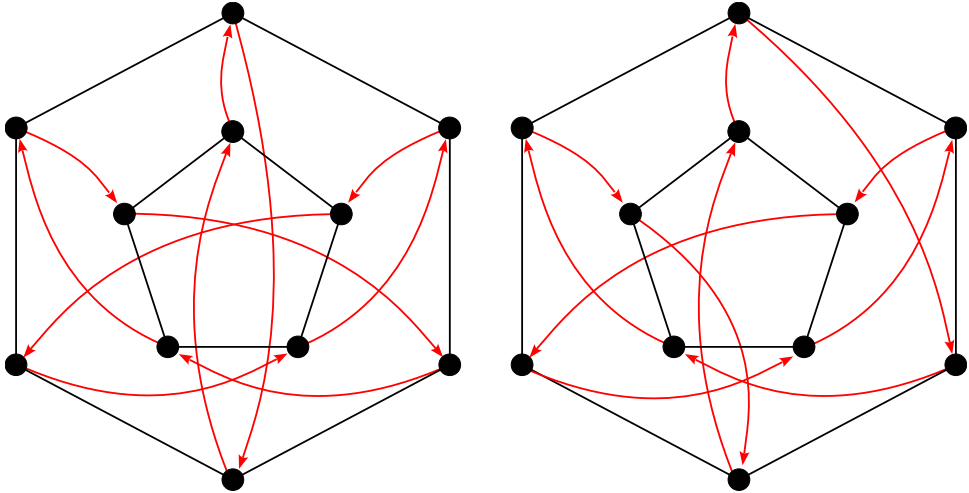}}
        \caption{The first two graphs $G_1$ and $G_2$ according to the status ranking.}
        \label{fig:three_best}
\end{figure}

\begin{proposition}
The mixed graphs $G_1$ and $G_2$ depicted in Fig. \ref{fig:three_best} have status vector $\mathbf{s}(G_1): 17^4,18^6,19^1$ and hence status norm $N_1(G_1)=18$. They are the first two ranked $(2,1,2)$-mixed radial Moore graphs according to the status norm.
\end{proposition}

\section{Existence of $(r,z,2)$-mixed radial Moore graphs}\label{sec:exist}

\subsection{A family of $(r,1,2)$-mixed radial Moore graphs.}\label{s:rz1k}

We begin by constructing the mixed Moore tree anchored at a distinguished vertex 0. The out-neighbors of $0$ is the set of integers $\{1,2,\dots,r+1\}$ being $r+1$ adjacent with an arc from $0$. The vertices at distance $2$ from the root are labelled with as an ordered pair of integers $(i,j)$ where $i$ indicates that the vertex is pending from vertex $i$ and $j$ is used to establish an order among the vertices pending from $i$. Note that the subtree that pends from vertex $r+1$ has an extra vertex compared with the other subtrees pending from $1$ to $r$ (see figure \ref{fig:MTz1}). 

\begin{figure}[ht]
        \centerline{\includegraphics*[width=\textwidth]{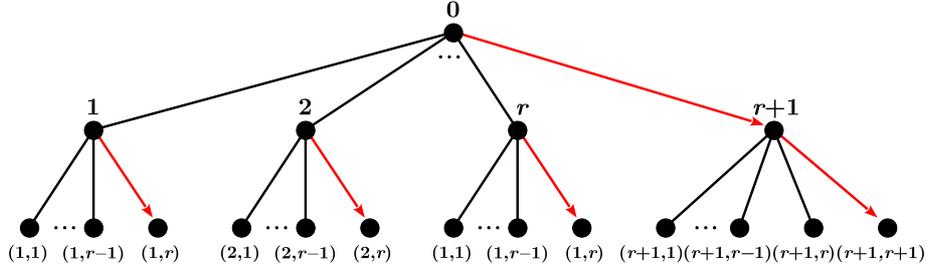}}
        \caption{Moore tree for $r>2$, $z=1$ and $k=2$, rooted at vertex $0$.}
        \label{fig:MTz1}
\end{figure}

The proposed graph family $H_r$ is constructed through the addition of the following edges to the Moore tree for $r>2$, $z=1$ and $k=2$:

\begin{enumerate}
    \item Edges $(i,j) \sim (u,v)$ such that
    \begin{equation*}
        j-i\equiv v-u \pmod{r+1},
    \end{equation*}
    where $i,u=1,2,\ldots,r+1$ and $j,v=1,2,\ldots,r$. Notice that  vertex $(r+1,r+1)$ is not included in this list.
    \item $(i,r) \sim (r+1,r+1)$ for every $i=1,\ldots,r$.
\end{enumerate}
The following arcs are also included in $H_r$:
\begin{enumerate}
\setcounter{enumi}{2}
    \item The arc $(r+1,r+1) \rightarrow 0$.
    \item Arcs $(r+1,i) \rightarrow r+1-i$, for every $i=1,\ldots,r$.
    \item Arcs from $(i,r)$ to $(r+1,r+1-i)$, for every $i=1,\ldots,r$.
    \item $(i,j) \rightarrow (r+i-j,r-j)$, for every $i=1,\ldots,r$ and $j=1,\ldots,r-1$ such that $j-i\not\equiv r \pmod{r+1}$, where indices in the last vertex must be taken modulus $r+1$.
    \item $(i,j) \rightarrow (i-1,j)$, for every $i=1,\ldots,r$ and $j=1,\ldots,r-1$ such that $j-i\equiv r \pmod{r+1}$. 
    
\end{enumerate}

\begin{figure}[ht]
        \centerline{\includegraphics[width=\textwidth]{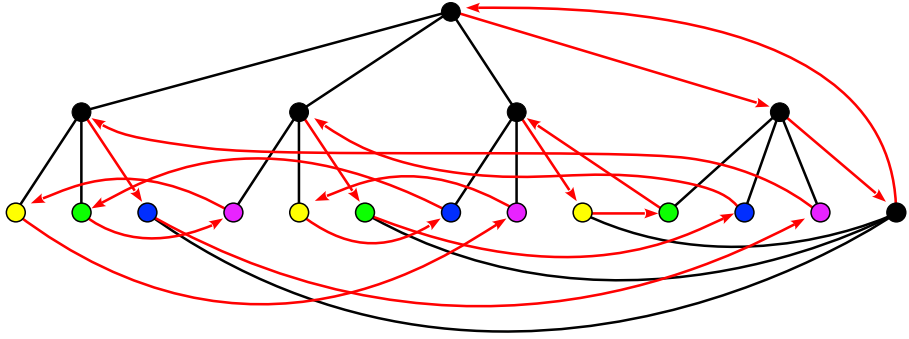}}
        \caption{Graph $H_3$. In order to keep the figure clean some adjacencies are indicated through colors (distinct from black). There is an edge joining vertices of the same color.}
        \label{H3}
\end{figure}

\begin{proposition}
Let $r>2$ be a positive integer. Then $H_r$ is a totally $(r,1)$-regular mixed graph.
\end{proposition}

\begin{proof}
Let us start proving the undirected regularity. In the Moore tree it is clear that vertices $i=0,1,\ldots,r+1$ have indeed undirected degree $r$. All vertices $(i,j)$ such that $i=1,\ldots,r+1$ and $j=1,\ldots,r-1$ and $(r+1,r)$ have undirected degree $1$. Finally $(i,r)$ vertices such that $i=1,\ldots,r$ and $(r+1,r+1)$ have undirected degree $0$. 
The edges added in step 2 increase the undircted degree of $(i,r)$ vertices such that $i=1,\ldots,r$ to $1$, and the undirected degree of $(r+1,r+1)$ to $r$. By now all vertices have degree $r$ but $(i,j)$ s.t. $i=1,2,\ldots,r+1$ and $j=1,2,\ldots,r$ which have undirected degree $1$. This set has cardinal $r(r+1)$. Precisely, the relationship stated in step 1 is an equivalence relation in this set and thus, it establishes a partition of $r+1$ classes, each one made up of $r$ vertices. For every element, an edge is drawn to the other members of the class, adding $r-1$ to the undirected degree, resulting a total of $r$.

The only vertex with indegree and outdegree equal to $1$ in the Moore tree is $r+1$. Step 3 completes the directed degrees of vertices $0$ and $(r+1,r+1)$. Vertices $i$ for $i=1,\ldots,r$ already have outdegree $1$, and the indegree turns $1$ after step 4. This step also increases by $1$ the outdegree of all vertices $(r+1,j)$. Step 5 completes the directed degrees of these vertices, since every $(r+1,j)$ receives an arc, and also completes vertices $(i,r)$ for $i=1,\ldots r$, since they already had indegree $1$ and after step 5 their outdegree is also $1$. It is easy to see that steps 6 and 7 make the outdegree of $(i,j)$ vertices s.t. $i=1,\ldots,r$ and $j=1,\ldots,r-1$ to be $1$, but it is less obvious that it also increases the indegree to $1$. To prove so, it is only necessary to show that there are not different arcs in steps 6 and 7 that are incident to the same vertex. Suppose that $(a,b)$ and $(c,d)$ with the conditions of step 6 are different vertices incident to the same vertex through the assignation of step 6. Then, 
\begin{equation*}
    r+a-b\equiv r+c-d\pmod{r+1}\quad\textrm{and}\quad r-b\equiv r-d\pmod{r+1}.
\end{equation*}
The second equivalence implies $b=d$ since $1 \leq b,d \leq r-1$ are integers that cannot differ by $r+1$. Introducing $b=d$ in the first equivalence entails $a\equiv c\pmod{r+1}$. Again this implies $a=c$, since $1 \leq a,c\ \leq r$ are integers that cannot differ by $r+1$. Thus $(a,b)$ and $(c,d)$ represent the same vertex, a contradiction that proves the statement. To end the proof, it is enough to see that the set of vertices that receive the arcs in step 6 are different from the ones that receive the arcs of step 7. In step 7 arcs are incident to vertices $(\alpha,\beta)=(i-1,j)$ s.t. $j-i\equiv r \pmod{r+1}$ and thus, they belong to the class of vertices such that $\beta-\alpha\equiv j-i+1 \equiv 0\pmod{r+1}$ while in step 6 arcs are incident to vertices $(\gamma,\delta)=(r+i-j,r-j)$ s.t. $j-i\not\equiv r \pmod{r+1}$ and thus $\delta-\gamma\equiv -i \not\equiv 0\pmod{r+1}$ since $1 \leq i \leq r$.
\end{proof}

\begin{proposition}
 For every $r>2$, $H_r$ is a $(r,1,2)$-mixed radial Moore graph with status vector:
\[
\begin{array}{ll}
\mathbf{s}(H_r): & (2r^2+3r+3)^{2},(2r^2+3r+4)^{r-1},(2r^2+4r+2)^{2},(3r^2+5)^{r^2-2r-1}, \\
 & (3r^2+6)^{5},(3r^2+7)^{2r-3},(3r^2+r+5)^{r-3},(3r^2+r+6)^{2}.  
\end{array}
\]
\end{proposition}

\begin{proof} 

The proof will proceed as follows. For every vertex $u\in V(H_r)$ we will find the number of vertices $v\in V(H_r)$ such that $d(u,v)\leq2$ trough a sketch of the first and second neighbors. Then, every vertex which is not present in the sketch $w\in V(H_r)$ such that $d(u,w)>2$ is listed and we proof that, indeed, $d(u,w)=3$ identifying the explicit $u-w$ path.


We will proceed in the mentioned manner, lumping vertices in subsets if the calculation can be made analogously. The number associated to a given edge in figures \ref{ES2}, \ref{ES9} and \ref{ES12} indicates the step in the construction of $H_r$ where this connection was added.

\begin{itemize}
    \item Vertex $0$ is clearly a central vertex, so it status is $\mathbf{s}_{r,1,2}=2r^2+3r+3$.

\begin{figure}[ht]
        \centerline{\includegraphics[width=\textwidth]{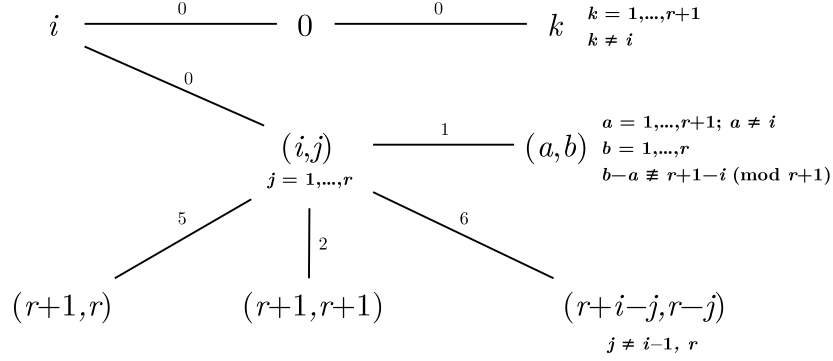}}
        \caption{Neighbors at distance two of vertex $i=2,\ldots,r$.}
        \label{ES2}
\end{figure}
    \item For vertices $i=1,\ldots,r$ the sketch is shown in figure \ref{ES2}. Vertex $1$ is a central vertex, but it is not the case for $i=2,\ldots,r$, where only vertex $(i-1,r)$ remains at distance three through the path:
     $$i\xrightarrow{\text{ 0 }}0\xrightarrow{\text{ 0 }}i-1\xrightarrow{\text{ 0 }}(i-1,r).$$
     Hence, this contributes to the status vector in $2r^2+3r+3,(2r^2+3r+4)^{r-1}$.
     
     \item For vertices $(i,r)$ s.t. $i=2,\ldots,r-1$ the sketch is shown in figure \ref{ES9}.
     
         \begin{figure}[h]
            \centerline{\includegraphics[width=\textwidth]{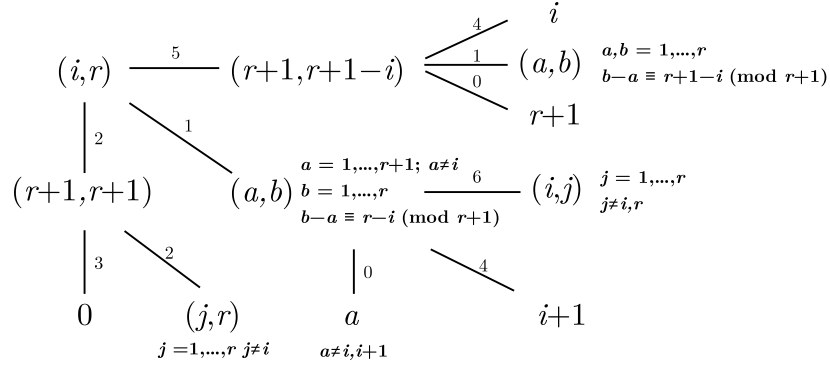}}
            \caption{Neighbors at distance one and two of vertex $(i,r)$ with $i=2,\ldots,r-1$.}
            \label{ES9}
        \end{figure}
        
        Remaining vertices:
        \begin{itemize}
            \item Vertex $(i,i)$:
            $$(i,r)\xrightarrow{\text{ 5 }}(r+1,r+1-i)\xrightarrow{\text{ 4 }}i\xrightarrow{\text{ 0 }}(i,i)$$
            
            \item Vertices $(r+1,j)$ s.t. $j=1,\ldots,r$ and $j\neq r-i,r+1-i$, a total of $r-2$ vertices:
            $$(i,r)\xrightarrow{\text{ 5 }}(r+1,r+1-i)\xrightarrow{\text{ 0 }}r+1\xrightarrow{\text{ 0 }}(r+1,j)$$
        
            \item Vertices $(i-1,j)$ s.t. $j=1,\ldots,r-2$, a total of $r-2$ vertices:
            $$(i,r)\xrightarrow{\text{ 1 }}(i-1,r-1)\xrightarrow{\text{ 0 }}i-1\xrightarrow{\text{ 0 }}(i-1,j)$$
        
            \item Vertices $(i+1,j)$ s.t. $j=2,\ldots,r$, a total of $r-2$ vertices:
            $$(i,r)\xrightarrow{\text{ 1 }}(r+1,r-i)\xrightarrow{\text{ 0 }}i+1\xrightarrow{\text{ 0 }}(i+1,j)$$
            
            \item Vertices $(\alpha,\beta)$ s.t. $\alpha=1,\ldots,r-1$; $\beta=1,\ldots,r-1$ and $\beta-\alpha\not\equiv r-i,r-i+1 \pmod{r+1}$ and $\alpha\neq i-1,i,i+1$, a total of $(r-3)(r-3)$ vertices:
            $$(i,r)\xrightarrow{\text{ 1 }}(\alpha,\alpha+r-i)\xrightarrow{\text{ 0 }}\alpha\xrightarrow{\text{ 0 }}(\alpha,\beta)$$
            
        \end{itemize}
        There is a total of $r^2-3r+4$ vertices at distance 3. So, they contribute in $(3r^2+7)^{r-2}$ to the status vector.

        \begin{figure}[h]
                \centerline{\includegraphics[width=\textwidth]{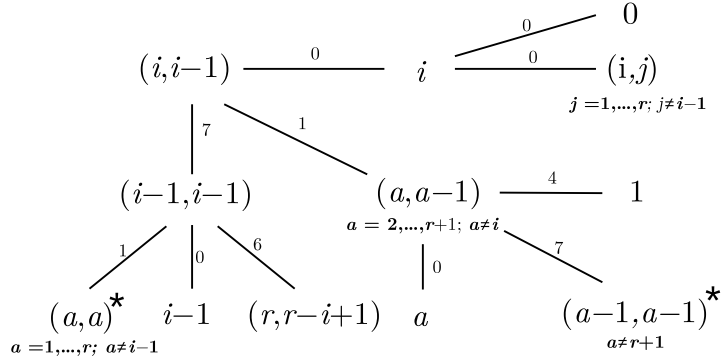}}
                \caption{Neighbors at distance one and two of vertex $(i,i-1)$ with $i=3,\ldots,r-1$. The stars point out the fact that there are repeated vertices in the tagged subsets.}
                \label{ES12}
        \end{figure}
        \item For vertices $(i,i-1)$ s.t. $i=3,\ldots,r-1$ the sketch is shown in Fig. \ref{ES12}.
        
        Remaining vertices:
        \begin{itemize}
            \item Vertex $(1,j)$ s.t. $j=2,\ldots,r$, a total of $r-1$ vertices:
            $$(i,i-1)\xrightarrow{\text{ 1 }}(r+1,r)\xrightarrow{\text{ 4 }}1\xrightarrow{\text{ 0 }}(1,j)$$
            
            \item Vertices $(r+1,j)$ s.t. $j=1,\ldots,r-1,r+1$, a total of $r$ vertices:
            $$(i,i-1)\xrightarrow{\text{ 1 }}(r+1,r)\xrightarrow{\text{ 0 }}r+1\xrightarrow{\text{ 0 }}(r+1,j)$$
            
            \item Vertices $(\alpha,\beta)$ s.t. $\alpha=2,\ldots,r-1$; $\beta=1,\ldots,r$ and $\beta-\alpha\not\equiv 0,r \pmod{r+1}$ and $(\alpha,\beta)\neq(r,r+1-i)$, a total of $(r-2)(r-2)-1$ vertices:
            $$(i,i-1)\xrightarrow{\text{ 1 }}(\alpha,\alpha-1)\xrightarrow{\text{ 0 }}\alpha\xrightarrow{\text{ 0 }}(\alpha,\beta)$$
            
        \end{itemize}
        There is a total of $r^2-2r+2$ vertices at distance 3. 
             
\end{itemize}
Analogously, one should proceed for every subset $V_i$ presented in table \ref{stab}. Column 3 contains the cardinal of each subset, and column 4 contains the number of vertex at distance $3$ from every vertex in $V_i$, denoted as $n_3(V_i)$. The collection of all subsets $V_i$ constitutes a partition of the vertex set. The status vector of $H_r$ can easily computed from this table, obtaining
\[
\begin{array}{ll}
\mathbf{s}(H_r): & (2r^2+3r+3)^{2},(2r^2+3r+4)^{r-1},(2r^2+4r+2)^{2},(3r^2+5)^{r^2-2r-1}, \\
 & (3r^2+6)^{5},(3r^2+7)^{2r-3},(3r^2+r+5)^{r-3},(3r^2+r+6)^{2}.  
\end{array}
\]
$N_1(H_r)$ can be also obtained from table \ref{stab} through $\displaystyle{N_1(H_r)=\sum_{0 \leq i\leq 9}{|V_i|\cdot n_3(V_i)}}$
obtaining $N_1(H_r)= r^4-2r^3+8r-2$.
\end{proof}

\begin{table}[htbp]
    \centering
    \begin{tabular}{|c|c|c|c|}\hline
      $i$ & $V_i$ & $|V_i|$& $n_3(V_i)$ \\\hline
      0 & $\{0,1\}$&  2 & 0\\
      1 & $\{i \ | i=2,\ldots,r\}$ & $r-1$ & 1\\
      2 & $\{r+1,(r+1,r+1)\}$&  2 & $r-1$\\
      3 & $\{(i,i) \ | \ i=2,\ldots,r-1\}$ & $r-2$ & $r^2-3r+2$\\
      4 & $\{(i,j){\neq}(1,r{-}1) \ | \ j{-}i\not\equiv 0,r\pmod{r{+}1}\}$ & $r^2-3r+1$ & $r^2-3r+2$\\
      5 & $\{(r+1,r),(1,r),(r,r),(1,1),(1,r-1)\}$ & 5 & $r^2-3r+3$\\
      6 & $\{(r+1,j) \ | \ j=1,\ldots,r-1 \}$ & $r-1$ & $r^2-3r+4$\\
      7 & $\{(i,r) \ | \ i=2,\ldots,r-1 \} $& $r-2$ & $r^2-3r+4$\\
      8 & $\{(i,i-1) \ | \ i\neq2,r\}$ & $r-3$ & $r^2-2r+2$\\
      9 & $\{(2,1),(r,r-1)\}$ & 2 & $r^2-2r+3$\\


        \hline
    \end{tabular}
    \caption{Partition of the vertices set of $H_r$ showing the number of vertices at distance $3$ for any given vertex of $V_i$.}
    \label{stab}
\end{table}

\subsection{An optimal family of $(1,z,2)$-mixed radial Moore graphs.}\label{s:r1zk}

Since mixed radial Moore graphs are presented as a kind of approximation of mixed Moore graphs, it is natural to think that their existence is guaranteed in the cases when a mixed Moore graph indeed exists. Moreover, the `closest' mixed radial Moore graph to being a mixed Moore graph could be obtained through a slight modification of the latter. 
\begin{figure}[htb]
\begin{center}
\begin{tabular}{ccc}
\includegraphics*[width=0.3\textwidth]{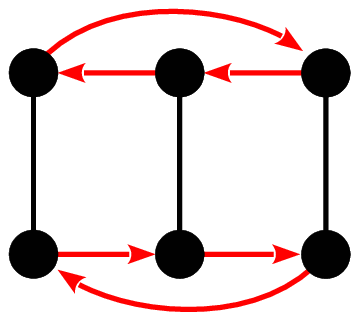} & \hspace{0.1\textwidth} &  \includegraphics*[width=0.3\textwidth]{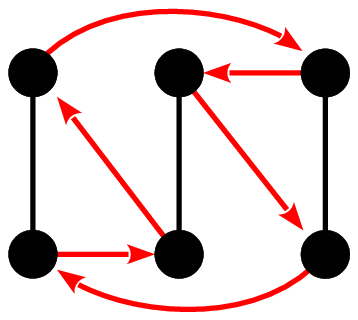} \\
$\mathbf{s}(G): 8^6$ & & $\mathbf{s}(G^*): 8^4,9^2$ \\
 & & $N_1(G^*)=2$
\end{tabular}
\caption{The unique mixed Moore graph $G$ on $6$ vertices and a mixed radial Moore graph $G^*$ obtained by an arc swap to $G$. Moreover, $G^*$ is the closest graph to $G$ according to the status norm.} \label{fig:kautz1}
\end{center}
\end{figure}

For instance, the mixed graph depicted in Fig. \ref{fig:kautz1} is precisely the unique mixed Moore graph for parameters $(r,z,k)=(1,1,2)$, which corresponds to the Kautz digraph on six vertices. It is not difficult to check that the mixed graph on its right is a mixed radial Moore graph and moreover, it is the closest mixed graph according to the status norm. This mixed graph has been obtained by a simple swap of two arcs of the Moore graph. In this new graph only two vertices have a slight modification of their status, incrementing just in one unit the value that they had in the original graph. The remaining vertices preserve the status that they had.

This particular result can be generalized. To see this, let $(u,u',v,v')$ be four (ordered) vertices of a mixed graph $D=(V,E,A)$ so that $uv,u'v' \in A$ and $uv',u'v \notin A\cup E$. The arc swap operation removes the two arcs $uv,u'v'$ and adds $uv'$ and $u'v$ to $A$ (Fig.~\ref{fig:swap}). Notice that this operation preserves the degrees of the vertices in the resulting mixed graph.

\begin{figure}[htb]
\begin{center}
\includegraphics*[scale=.5]{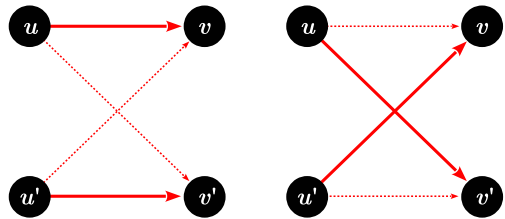} 
\caption{Arc swap operation.} \label{fig:swap}
\end{center}
\end{figure}

It turns out that, under certain conditions, every mixed graph constructed by an arc swapping to a proper mixed Moore graph of diameter $2$ is a $(r,z,2)$-mixed radial Moore graph, and moreover, this mixed graph is the closest graph to being a Moore graph.

\begin{theorem}\label{th:exist}
Let $G^*$ be the mixed graph obtained by an arc swapping of some vertices $(u,u',v,v')$ in a proper mixed Moore graph $G$ of diameter $2$, such that $\Gamma^-(u)=\Gamma^-(u')$ and $\Gamma^+(v)=\Gamma^+(v')$. Then $G^*$ is a $(r,z,2)$-mixed radial Moore graph with $N_1(G^*)=2$.
\end{theorem}

\begin{proof}
Assume that $(u,u',v,v')$ are four different vertices such that $\Gamma^-(u)=\Gamma^-(u')$ and $\Gamma^+(v)=\Gamma^+(v')$. Let us see what happens with the shortest paths in $G^*$. We divide the proof according to the length of the shortest paths in $G$:
\begin{itemize}
 \item Let $w_0,w_1,w_2$ be a shortest path (of length $2$) in $G$. Once the arc swap $(u,u',v,v')$ is done, we will see that either the same path $w_0,w_1,w_2$ is a shortest path in $G^*$ or there is an alternative path with length $\leq 2$. For the case $w_0=u$ [$w_0=u'$] and $w_2=v'$ [$w_2=v$], then the arc swap produces a shortest path. For the remaining cases, we will see that the length of the path is preserved. Indeed, if $w_i \neq u,u'$ for $i=0,1$ then the same path is a shortest path in $G^*$. If $w_1=u$ and $w_2 \neq v$ the same path is also a shortest path in $G^*$, and the shortest path $w_0,u,v$ in $G$ becomes $w_0,u',v$ in $G^*$ since $w_0 \in \Gamma^-(u)$ and by assumption $\Gamma^-(u)=\Gamma^-(u')$. The same argument applies when $w_1=u'$ by exchanging $u$ by $u'$. It remains to consider the case when $w_0=u$ or $w_0=u'$ (again, by symmetry, we can consider just one case, say $w_0=u$): when $w_1 \neq v,v'$ the same shortest path is valid and for $w_1=v$ or $v'$ we obtain a shortest path in $G_z$ just exchanging $v$ by $v'$ (due to the arc swapping) since $\Gamma^+(v)=\Gamma^+(v')$.
\item For shortest paths of length one, we only have to consider paths $u,v$ and $u',v'$ (of length $1$) in $G$, which no longer exist in $G^*$. We will prove that there is a path of length three from $u$ to $v$ and there is no one of length $2$ (the proof for the case $u',v'$ follows the same ideas). Since $\Gamma^+(v)=\Gamma^+(v')$ and $r\geq 1$, then it exists $w \in \Gamma^+(v)$ such that $vw$ is an edge and $v'w$ is an arc. So $u,v',w,v$ is the desired path. Assume that there is a path of length $2$ in $G^*$ from $u$ to $v$, then it exists $w \in \Gamma^+(u)\cap \Gamma^-(v)$. Notice that this is only possible if $w=v'$ or $w=u'$, since otherwise the shortest path from $u$ to $v$ in $G$ would have length $2$. But if the path $u,v',v$ exists in $G^*$ then we would have the arc $v'v$ in $G$ contradicting the fact that $\Gamma^+(v)=\Gamma^+(v')$. The same applies to the case $u,u',v$ where we would have a contraction with $\Gamma^-(u)=\Gamma^-(u')$.
\end{itemize}
Notice that under the assumption $\Gamma^-(u)=\Gamma^-(u')$ and $\Gamma^+(v)=\Gamma^+(v')$ there are only two paths of length $3$ in $G^*$, those starting at $u$ and $u'$ and ending at $v$ and $v'$, respectively. As a consequence, every vertex in $G^*$ is a central vertex, except $u$ and $u'$, having just one vertex at distance $3$ each, that is, $G^*$ has radius $2$, diameter $3$ and $N_1(G^*)=2$. 
\end{proof}

Theorem \ref{th:exist} can be applied at least at the unique infinite family of proper mixed Moore graphs known until now.
The family of {\em Kautz digraphs}. The vertices of the {\em Kautz digraph} $\ka(d,k)$, $d \geq 1$, $k \geq 1$, are words of length $k$ on the alphabet $\Sigma=\{0,1,\dots,d\}$ without two consecutive identical numbers. There is an arc from vertex $(v_0,v_1,\dots,v_{k-1})$ to vertices $(v_1,\dots,v_{k-1},x)$, where $x \in S \setminus \{v_{k-1}\}$. It is known that $\ka(d,k)$ has order $(d+1)d^{k-1}$ and diameter $k$. Kautz digraphs $\ka(d,2)$ are the unique mixed Moore graph of diameter $2$ for every $d \geq 2$ (taking into account that every oriented digon is replaced by an edge, that is, $r=1$ and $z=d-1$) as noted in \cite{NMG07}. \\

We may construct a new family of $(1,z,2)$-mixed radial Moore graphs by performing a convinient {\em arc swap} to $\ka(1+z,2)$ such that the conditions of theorem \ref{th:exist} are fullfilled. Moreover, this is the best possible, as we will see later. First we provide a helpful lemma. 

\begin{lemma}\label{lem:lema1}
Let $G$ be a $(r,z,k)$-mixed radial Moore graph and $w\in V(G)$ be a central vertex of $G$. Suppose $w$ is adjacent to a non-central vertex $u\in V(G)$ through an edge. Then there exists another vertex in the out-neighborhood of $w$ different from $u$ which is non-central.
\end{lemma}

\begin{proof}

Let $\Gamma^+(w)=\{a_1,\dots,a_z,e_1,\dots,e_{r-1},u\}$ be the out-neighborhood of $w$, such that vertices $a_i$ are adjacent from $w$ through an arc and vertices $e_j$ are adjacent from $w$ through an edge. By contradiction suppose all vertices in $\Gamma^+(w)\setminus\{u\}$ are central vertices.

Since $u$ is a non-central vertex and $G$ is a $(r,z,k)$-mixed radial Moore graph there must exist a vertex $v\in V(G)$ such that $d(u,v)=k+1$ and thus $d(w,v)=k$. The distance preserving spanning tree rooted at $w$ must be a Moore tree, and $v$ must be located at distance $k$, belonging to any of the subtrees rooted at $\Gamma^+(w)\setminus\{u\}$. Without loss of generality, we can suppose that it hangs from the subtree rooted at $a_1$. Since all vertices in $\Gamma^+(w)\setminus\{u\}$ are central there must be a path of length not greater than $k$ from every $a_i$ and $e_i$ to $v$. For $a_1$ it already exists a $(k-1)$-path in the Moore Tree but it is forced that for every vertex $a_i$ [resp. {\em $e_i$}] in $\Gamma^+(w)\setminus\{u,a_1\}$ there exists a vertex $a'_i$ [resp. {\em $e'_i$}] such that $d(a_i,a'_i)=k-1$ [resp. {\em $d(e_i,e'_i)=k-1$}] and there is an adjacency (arc or edge) from $a_i$ [resp. {\em $e_i$}] to $v$ (see figure \ref{fig:proof}).

\begin{figure}[ht]
        \centerline{\includegraphics*[width=\textwidth]{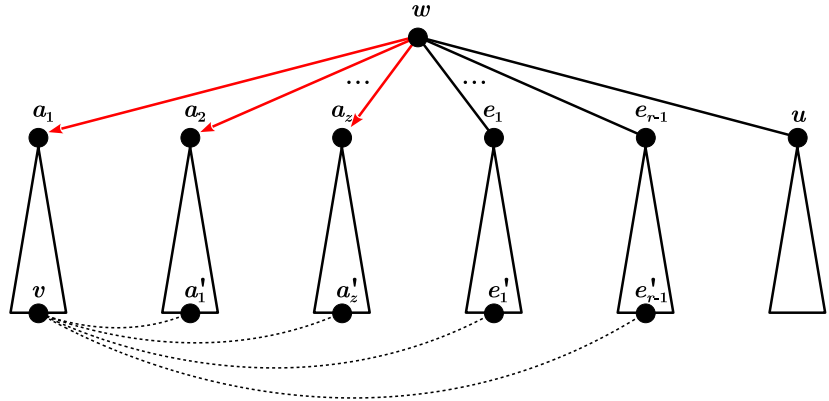}}
        \caption{Distance preserving spanning tree rooted at $w$. Dashed lines are used to represent adjacencies that might be either arcs or edges.}
        \label{fig:proof}
\end{figure}
These forced adjacencies increase in $r+z-2$ the in-degree of $v$, and counting the arc or edge that keeps $v$ hanging from the tree it results that the total in-degree of $v$ is $r+z-1$. Since $G$ is a totally $(r,z)$-regular graph, there must be another vertex $v'$ adjacent to $v$ but:
\begin{itemize}
    \item Vertex $v'$ must hold $d(w,v')=k$, otherwise $w$ would not be a central vertex.
    \item It cannot pend from the subtree rooted at $u$. This would mean $d(u,v)=k$, which would be a contradiction.
    \item It cannot pend from a vertex at distance $k-1$ from $a_1$, since $a_1$ is a central vertex and there would be two different $(a_1,v)$-paths of length less than $k$.
    \item It cannot pend from vertex $a_i$ [resp. {\em $e_i$}] in $\Gamma^+(w)\setminus\{u,a_1\}$, since $a_i$ [resp. {\em $e_i$}] is a central vertex and there would be two different $(a_i,v)$-paths [resp. {\em $(e_i,v)$-paths}]  of length $k$.
\end{itemize}

\end{proof}


\begin{corollary}
Let $G_z$ be a mixed graph obtained by an arc swapping of vertices $01,03,12,32$ in $\ka(z+1,2)$, $z\geq2$. Then $G_z$ is a $(1,z,2)$-mixed radial Moore graph, which is the closest graph to a mixed Moore graph (according to the status norm). 
\end{corollary}

\begin{proof}
Clearly $\Gamma^-(01)=\Gamma^-(03)=\{x0 \, | \, x \in \Sigma,x\neq0\}$ and $\Gamma^+(12)=\Gamma^+(32)=\{2y \, | \, y \in \Sigma,y\neq 2\}.$ By theorem \ref{th:exist} we have that $G_z$ is a $(1,z,2)$-mixed radial Moore graph with $N_1(G_z)=2$. According to lemma \ref{lem:lema1}, every mixed radial Moore graph $G$ has at least $2$ non-central vertices, so $N_1(G) \geq 2$, and the result follows.
\end{proof}

\begin{figure}[ht]
        \centerline{\includegraphics[width=\textwidth]{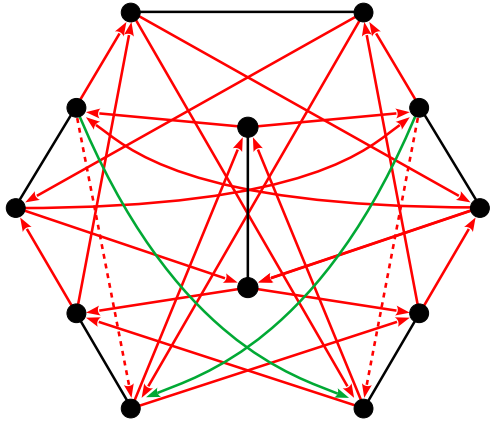}}
        \caption{The digraph $G_z$ for $z=2$, obtained from the Kautz digraph $Ka(3,2)$ (the dashed arcs are replaced by the green ones in the optimal swap).}
        \label{ka32}
\end{figure}

For the undirected case, every infinite family of radial Moore graphs of diameter $2$ constructed until know has a norm status that polinomialy increases with the degree. The family of mixed radial Moore graphs $G_z$ has a constant norm status equal to $2$, independently of $z$. This family of mixed graphs is cospectral with the Kautz digraphs and preserve some symmetries, as the one depicted in Fig. \ref{ka32}.

\section{Final remarks and open problems}\label{sec:remark}

Mixed radial Moore graphs are extremal graphs that are interesting by its own definition, but they are also a nice approximation to the well known Moore graphs. In this paper we provide two infinite families $H_r$ and $G_z$ of mixed radial Moore graphs of diameter $2$, but it would be very interesting to know if they exist for any combination of the parameters $(r,z,k)$, even for diameter greater than two.

\begin{question}
Are there $(r,z,k)$-mixed radial Moore graphs for any $r,z \geq 1$ and $k \geq 2$?
\end{question}

The arc swap method used in subsection \ref{s:r1zk} provides an optimal family $G_z$ of $(1,z,2)$-mixed radial Moore graphs, but in could be useful for other set of parameters $(r,z,k)$ where mixed Moore graphs exist. For instance, we have applied this method to the unique mixed Moore graph on $18$ vertices, the one for the case $(3,1,2)$ (described in \cite{B79}). This mixed Moore graph $B_{18}$ has an status vector of $\mathbf{s}(B_{18}):30^{18}$. It turns out that an optimal arc swap produces a mixed graph $B'_{18}$ with status vector $\mathbf{s}(B'_{18}):30^{10},31^6,32^2$, that is, $N_1(B'_{18})=10$. It is known that $B_{18}$ contains many copies of the Kautz digraph on $6$ vertices (see \cite{Lopez2015522}). It turns out that $B'_{18}$ is constructed from $B_{18}$ by the arc swap depicted in Fig. \ref{fig:kautz1} to any of this inside copies, that is, the optimal arc swap to the family of kautz digraphs described in subsection \ref{s:r1zk} produces the closest mixed radial Moore graph to $B_{18}$ so far (among all possible arc swaps that in can be done). Nevertheless, we do not know if $B'_{18}$ is the top ranked graph in ${\cal RM}(3,1,2)$, since we only have checked those graphs obtained by a single arc swap to $B_{18}$. This method has been applied also to the two mixed Moore graphs of order $108$ known until know (see \cite{J12}) with status vector $\mathbf{s}_{3,7,2}:204^{108}$, obtaining a mixed radial Moore graph with status vector $204^{98},205^{8},209^{2}$, that is, norm status equal to $18$.

\begin{problem}
Find those mixed radial Moore graphs best ranked in ${\cal RM}(r,z,k)$ for specific values $r,z \geq 1$ and $k \geq 2$, in those cases where mixed radial Moore graphs exist.
\end{problem}

We have answered to this question for the cases $(2,1,2)$ and $(1,z,2)$ in this paper, but there are infinitely many open cases. So, table \ref{Results} is presented below with the best results found so far.

\begin{table}[htbp]
    \centering
    \begin{tabular}{|c|cccccccc|}\hline
    $z$\textbackslash$r$ & 1 & 2 & 3 & 4 & 5 & 6 & 7 & $\dots$\\ \hline 
    1 & \textcolor{red}{2} & \textcolor{red}{8} & \textcolor{blue}{10} & \textcolor{green}{158} & \textcolor{green}{413} & \textcolor{green}{910} & \textcolor{green}{1769}  & \textcolor{green}{$\dots$}\\
    2 & \textcolor{red}{2} & ? & ? & ? & ? & ? & ? & $\dots$\\
    3 & \textcolor{red}{2} & ? & ? & ? & ? & ? & \textcolor{blue}{18} & $\dots$\\
    4 & \textcolor{red}{2} & ? & ? & ? & ? & ? & ? & $\dots$\\
    $\vdots$ & \textcolor{red}{$\vdots$} & $\vdots$ & $\vdots$ & $\vdots$ &$\vdots$ & $\vdots$ & $\vdots$ & \\ \hline 
    \end{tabular}
    \caption{Values of $N_1(G)$ of the closest mixed radial Moore graphs $G$ found so far. Red numbers correspond to optimal values. Blue numbers show the status norm of the closest mixed radial Moore graphs obtained by arc swapping to a Moore graph. Green values are the status norm values of the $H_r$ family.}
    \label{Results}
\end{table}

\subsection*{Acknowledgements}

The authors would like to thank “Ministerio de Ciencia e Innovación” of Spain, MCIN/AEI/10.13039/501100011033 (grant references PID2019-111536RB-I00 and PID2020-115442RB-I00) and AGAUR (grants references 2017SGR1158 and 2017SGR1276). Research of J. M. Ceresuela was supported by Secretaria d'Universitats i Recerca del Departament d'Empresa i Coneixement de la Generalitat de Catalunya (grant 2020 FISDU 00596). D. Chemisana thanks “Institució Catalana de Recerca i Estudis Avan\c cats (ICREA)” for the ICREA Acadèmia award.
\bibliographystyle{plain}
\bibliography{biblio}

\end{document}